\documentclass[letterpaper, 10 pt, conference]{ieeeconf}  

\IEEEoverridecommandlockouts    

\usepackage{graphics} 
\usepackage{amsmath} 
\usepackage{amssymb}  
\usepackage[FIGTOPCAP]{subfigure}
\usepackage{graphicx}
\newcommand{\X}{\mathbb{X}}
\newcommand{\R}{\mathbb{R}}
\newcommand{\N}{\mathbb{N}}

\newcommand{\D}{\mathbb{D}}
\newcommand{\K}{\mathcal{K}}

\usepackage[dvipsnames]{xcolor}

\newtheorem{defn}{Definition}
\newtheorem{prop}{Proposition}

\newtheorem{remark}{Remark}%

\title{\LARGE \bf Error bounds on analytic Koopman-based Lyapunov functions
}

\author{François-Grégoire Bierwart$^{1}$ and Alexandre Mauroy$^{2}$
\thanks{$^{1,2}$ Department of Mathematics and Namur Research Institute for Complex Systems (naXys), University of Namur, 5000, Belgium
        {\tt\small francois-gregoire.bierwart.@unamur.be, alexandre.mauroy@unamur.be}}%
}

\begin{document}

\maketitle
\thispagestyle{empty}
\pagestyle{empty}

\begin{abstract}
The Koopman operator provides an infinite-dimensional linear description of nonlinear dynamical systems that can be leveraged in the context of stability analysis. In particular, Lyapunov functions can be obtained in a systematic way via the eigenfunctions of the Koopman operator. However, these eigenfunctions are computed from finite-dimensional approximations, resulting in approximated Lyapunov functions that must be validated.
In this paper, we provide theoretical error bounds on the approximation of the eigenfunctions of the Koopman operator in the case of analytic vector field and finite-dimensional approximation in polynomial subspaces. We leverage these results to assess the validity of Koopman-based Lyapunov functions and obtain an optimization-free inner approximation of the region of attraction of an equilibrium.
\end{abstract}

\section{Introduction}

Stability analysis of nonlinear dynamics is crucial to characterize the long-term behavior of many natural phenomena and artificial systems. Such analysis mostly relies on the existence of a Lyapunov function \cite{c2}, whose design is usually challenging. This has lead to the emergence of various numerical methods to compute Lyapunov functions, which are based, for instance, on sum-of-squares polynomial optimization \cite{c18}, Zubov's equation \cite{c16}, or piecewise linear functions \cite{c17}, to list a few. We refer the reader to \cite{c15} for an overview. 

More recently, the Koopman operator framework has appeared as a valid alternative approach to stability analysis. This framework provides a linear description of the original nonlinear system in term of the evolution of \textit{observable} functions. Linearity can therefore be leveraged for stability analysis, and in particular, the eigenfunctions of the Koopman operator can be used to construct Lyapunov functions in a systematic way \cite{c3}. However, since the operator is defined on an infinite dimensional space, the eigenfunctions are typically approximated in a finite-dimensional subspace spanned by a set of \textit{basis functions} (see e.g. \cite{c4} for a review). This results in approximated candidate Lyapunov functions that must be validated.

While numerical validation schemes relying on SOS-based methods have been proposed in previous works \cite{c6,c19}, there is no theoretical contribution on the validity of Lyapunov functions based on the approximated eigenfunctions of the Koopman operator (even though theoretical results have been proposed recently in a similar context for Gramian-based Lyapunov functions \cite{c16}). We fill this gap in this paper by presenting theoretical error bounds on the approximation of the Lyapunov function and its time derivative. Focusing on the case of monomial basis functions and analytic vector fields, we prove the convergence of the approximated Lyapunov function under the assumption of an analytic vector field. Moreover, the obtained error bounds provide inner approximations of the region of attraction (ROA) of the equilibrium point based on level sets of the candidate Lyapunov function.  Yet, our estimation does not involve optimization techniques such as SOS methods \cite{c18}.


The paper is organized as follows. In Section \ref{preli}, we describe the main setting and the construction of Lyapunov function from the eigenfunctions of the Koopman operator. Upper bounds on the approximation error of the eigenfunctions are provided in Section \ref{sec:error_phi}, while upper bounds on the approximation of the Lyapunov function and its time derivative are obtained in Section \ref{stab_sec}. An inner approximation of the ROA is also derived, which is illustrated with two numerical examples in Section \ref{example}. Finally, concluding remarks and perspectives are given in Section \ref{conclusion}.

\section{Preliminaries}\label{preli}

Consider the dynamical system
\begin{equation}\label{eq}
\dot{x} = F(x), \quad x\in\X
\end{equation}
where $\X\subset\R^n$ is a compact set and the vector field $F$ is analytic and Lipschitz continuous. We denote by \mbox{$\varphi^{t}(x) : \R^+ \times \mathbb{X} \rightarrow \mathbb{X}$} the flow map generated by (\ref{eq}). We assume that the system (\ref{eq}) admits a locally stable hyperbolic equilibrium at the origin and that the eigenvalues of the Jacobian matrix $\mathbf{J}_F(0)$ of $F$ evaluated at $0$ are non-resonant. The classical approach to study the global stability of \mbox{$x^* = 0$} relies on the design of a Lyapunov function $V$, i.e. a positive function decreasing along the system trajectories \cite{c2}. In this paper, we will rely on the results of \cite{c3}, where spectral properties of the Koopman operator are leveraged to construct Lyapunov functions in a systematic way.

\subsection{Koopman operator framework for stability analysis}

In this section, we briefly describe the Koopman operator framework for stability analysis of dynamical systems.

\begin{defn}
Let $\mathcal{F}$ be a Banach space of observable functions defined on $\mathbb{X}$. We define the Koopman semigroup as the family $\{\K_t\}_{t\geq 0}$, $\K_t : \mathcal{F} \rightarrow \mathcal{F}$ such that
\mbox{$
\K_t f ~=~ f \circ \varphi^{t}, \quad \forall f \in \mathcal{F}.
$}
\end{defn}
The semigroup $\{\K_t\}_t$ consists of linear operators, so that we can define their eigenvalues and eigenfunctions. We define $\phi_{\lambda}$ as an eigenfunction of $\K_t$ (for some $t\geq 0$) associated with the eigenvalue $e^{\lambda t}$ if 
$\K_t\phi_{\lambda} = e^{\lambda t}\phi_{\lambda}.$
If $\mathcal{F} = C(\X)$, the semigroup of Koopman operator is strongly continuous and we can define its infinitesimal generator \vspace{-0.1cm} 

$$
\mathcal{L}f ~ \triangleq ~ \lim_{t \rightarrow 0^{+}}\dfrac{\mathcal{K}_tf - f}{t} ~=~ F\cdot\nabla f\vspace{0.1cm}
$$

\noindent for all $f\in C^1(\X)$. One can easily show that $\mathcal{L}\phi_{\lambda} ~ = ~\lambda \phi_{\lambda}$. Since the expression of the vector field is known, we will rather use the infinitesimal generator rather than the semigroup.

The eigenfunctions of $\mathcal{L}$ associated with eigenvalues with strictly negative real part capture stability properties of \eqref{eq}. Indeed, it is shown in \cite{c3} that, if the origin is an hyperbolic stable equilibrium point over a (invariant) region $D \subset \X$, there are $n$ \emph{principal eigenfunctions} $\phi_{\lambda_i} \in C^1(D)$, \mbox{$i = 1,\ldots,n$,} such that $\Re(\lambda_i) < 0$ where $\lambda_i$ are the eigenvalues of $\mathbf{J}_F(0)$. These eigenfunctions allow to construct the following generic Lyapunov function given by 
\begin{equation}\label{TrueLyap}
V ~=~ \sum_{i=1}^n\left|\phi_{\lambda_i}\right|^2,~~~\forall x\in D.
\end{equation}

\subsection{Approximation of the Lyapunov function}

The infinitesimal generator $\mathcal{L}$ is defined on an infinite-dimensional space, so that the computation of its eigenfunctions is not straightforward. A common way to circumvent this issue is to compute a finite-dimensional approximation of $\mathcal{L}$ (see e.g. \cite[Chapter 1]{c4} for an overview). Let $\mathcal{F}_{\ell}$ be a finite-dimensional subspace of $\mathcal{F}$ spanned by a set of $\ell \in \N_0$ basis functions $\{\psi_j\}_{j=1}^{\ell}$. An approximation of the infinitesimal generator $\mathcal{L}$ is given by 
$
\mathcal{L}_{\ell} ~:=~ \Pi\,\mathcal{L}\hspace{-0.12cm}\mid_{\mathcal{F}_{\ell}} ~:~ \mathcal{F}_{\ell} \rightarrow \mathcal{F}_{\ell}.
$
The approximation $\mathcal{L}_{\ell}$ is a finite-dimensional operator that can be represented by a matrix $\mathbf{L}_{\ell} \in \R^{{\ell}\times {\ell}}$ of $\mathcal{L}$
whose $i$-th column $\mathbf{c}_i$ contains the coefficient of the expansion of $\mathcal{L}_{\ell}\psi_i$ in the basis. The eigenfunctions $\widetilde{\phi}_{\,\widetilde{\lambda}}$ of $\mathcal{L}_{\ell}$, associated with the eigenvalues $\widetilde{\lambda}$, are given by the right eigenvectors $v$ of $\mathbf{L}_{\ell}$, i.e. $\widetilde{\phi}_{\,\widetilde{\lambda}}(x) = v^T \Psi(x)$ where $\Psi(x) \triangleq [\psi_1(x),\ldots,\psi_N(x)]^\top$ (see \cite{c4} for more details). They can be used to approximate the true principal eigenfunctions $\phi_{\lambda_i}$, provided that $\lambda_i \approx \widetilde{\lambda}_i$, and therefore to construct a candidate Lyapunov function of the form \eqref{TrueLyap}. In what follows, we will denote by $\widetilde{V}$ the Lyapunov candidate computed with $\widetilde{\phi}_{\,\widetilde{\lambda}_i}$.

There is no guarantee that the approximated eigenfunctions approximate well those of $\mathcal{L}$. However, it follows from the Poincaré linearization theorem that there exist analytic eigenfunctions in some neighborhood $\D^n(S) := \{x\in\R^n \mid ~|x_i| < S ~\forall i\}$ of the origin (for some $S>0$), provided that the eigenvalues of $\mathbf{J}_F(0)$ are non-resonant \cite{c12}. Note that the approximation region $\mathbb{D}^n(S)$ might be conservative depending on the shape of the ROA and the existence of unstable equilibrium points close to the origin. Since the eigenfunctions are analytic over $\D^n(S)$, a natural choice for $\mathcal{F}_\ell$ is the subspace of monomials up to order $N \in\mathbb{N}$, with $\ell = \frac{(n+N)!}{n!N!}$. In this case, it is well-known that $\sigma(\mathbf{J}_F(0)) \subset \sigma(\mathbf{L_\ell})$ and that $\widetilde{\phi}_{\lambda} = \Pi \phi_{\lambda}$ for any analytic $\phi_\lambda$, where $\Pi := P_N$ is the truncation operator defined as 
$$
P_N(x_1^{k_1}\ldots\,x_n^{k_n}) = 
\begin{cases}
     x_1^{k_1}\ldots\,x_n^{k_n} & \text{if } \sum_i k_i \leq N,\\
     0 & \text{otherwise}. 
\end{cases}
$$

Note that this property follows from the fact that
\begin{equation}
\label{eq:PLP}
    P_N\mathcal{L}P_N  = P_N\mathcal{L}.
\end{equation}
This implies that analytic (principal) eigenfunctions are approximated by their truncated Taylor series, so that $P_N\phi_{\lambda_i}(x) \rightarrow \phi_{\lambda_i}(x)$ as $N\rightarrow\infty$ for all $x\in\D^n(S)$ and $i=1,\ldots,n$. Moreover, the exact projection $P_N \phi_{\lambda}$ can be easily computed by exploiting the triangular structure of $\mathbf{L}_\ell$ (see e.g. \cite{c14}).


\section{Error bounds on the approximation of the eigenfunctions}\label{sec:error_phi}

In this section, we derive error bounds on the Taylor approximation of the eigenfunction 
$$
\widetilde{\phi}_{\lambda} = P_N \phi_{\lambda}(x) = \sum_{\substack{k\in\N_0^n \\ |k|\leq N}} c_k\, x^k
$$
where $c_k$ are the Taylor coefficients of the expansion and where $k\in\N_0^n$ refers to a multi-index notation, with \mbox{$|k| = k_1+\cdots+k_n$} and $x^k = x_1^{k_1}\ldots\,x_n^{k_n}$. Note that the results developed below are valid for any analytic function. Let us first consider the following general result. 
\begin{prop}[\cite{c13}]\label{thm:error_phi1}
Let $\phi_{\lambda}$ be an analytic function over $\D^n(S)$ such that 
$
\sup_{x\in\overline{\D^n(S)}}|\phi_{\lambda}(x)| \leq M
$. Then, \mbox{for any $x\in \D^n(R)$,}
\begin{equation}\label{eq:error_phi1}
\left|\phi_{\lambda}(x)-(P_N \phi_{\lambda})(x)\right| \leq \dfrac{M}{1-R/S}\left(\dfrac{R}{S}\right)^{N+1},
\end{equation}
where $0<R<S$.
\end{prop}
This result can be used with some a priori bound on the eigenfunction. However, such bound is not known. As an alternative, we derive the following result, which relies on the Taylor coefficients of the eigenfunction. 
\begin{prop}\label{thm:error_phi2}
Let $\phi_{\lambda}$ be an analytic function over $\D^n(S)$ for some $S>0$ and such that its Taylor coefficients satisfy $ \max_{|k|>N}|c_k|S^k < M_1(N) \in \R^+$. Then, for all $x\in \D^n(R)$ with $0<R<S$,
\begin{equation}\label{eq:error_phi2}
\begin{split}
& |\phi_{\lambda}(x)-(P_N\phi_{\lambda})(x)| \\ 
& \quad < M_1(N)\left[\left(\dfrac{1}{1-R/S}\right)^n - \displaystyle\sum_{k=0}^{N}\left(\!\!{n\choose k}\!\!\right)\left(\dfrac{R}{S}\right)^{k}\right].
\end{split}
\end{equation}
\end{prop}
\begin{proof}
We have 
\begin{equation*}
\left|\phi_{\lambda}- P_N\phi_{\lambda}\right| = \left|\displaystyle\sum_{|k|>N}c_kx^k\right|
 \leq  \displaystyle\sum_{|k|>N}|c_k||x|^{k}.
\end{equation*}
For $x \in \D^n(R)$, since $|c_k|S^{k} < M_1(N) $ for all $|k|>N$, it holds 
$$
\begin{array}{rcl}
\displaystyle\sum_{|k|>N}|c_k||x|^{k} &\leq&
\displaystyle\sum_{|k|>N}|c_{k}|R^k,\\
&=& \displaystyle\sum_{|k|>N}|c_{k}|R^k\left(\dfrac{S}{S}\right)^{k},\\
& \leq & M_1(N)\displaystyle\sum_{|k|>N}\left(\dfrac{R}{S}\right)^{k}.\\
\end{array}
$$
Moreover, we can write 
\begin{align}
M_1(N)\displaystyle\sum_{|k|>N}\left(\dfrac{R}{S}\right)^{k} &= M_1(N)\displaystyle\sum_{k=N+1}^{\infty}\dbinom{n+k-1}{k}\left(\dfrac{R}{S}\right)^{k}, \nonumber \\
& ~\triangleq~ M_1(N)\displaystyle\sum_{k=N+1}^{\infty}\left(\!\!{n\choose k}\!\!\right)\left(\dfrac{R}{S}\right)^{k}\label{binomial}.
\end{align} 

\vspace{-0.15cm}
\noindent Since the generating function of the multiset coefficient leads to \vspace{-0.1cm}
$$
\displaystyle\sum_{k=0}^{\infty}\left(\!\!{n\choose k}\!\!\right)\left(\dfrac{R}{S}\right)^{k} = \left(\dfrac{1}{1-R/S}\right)^n,\vspace{0.1cm}
$$
we finally obtain \vspace{-0.2cm}
\begin{align*}
& M_1(N)\displaystyle\sum_{|k|>N}\left(\dfrac{R}{S}\right)^{k} \\
& \quad = M_1(N)\left[\left(\dfrac{1}{1-R/S}\right)^n - \displaystyle\sum_{k=0}^{N}\left(\!\!{n\choose k}\!\!\right)\left(\dfrac{R}{S}\right)^{k}\right].
\end{align*} \vspace{-0.5cm}

\end{proof}
Finally, we present another result which is similar in spirit to Proposition \ref{thm:error_phi2}, but relies on the Cauchy-Schwarz inequality. 
\begin{prop}\label{thm:error_phi3}
Let $\phi_{\lambda}$ be an analytic function over $\D^n(S)$ for some $S>0$ such that its Taylor coefficients satisfy
$
\sum_{|k|>N}|c_k|^2S^{2k} < M_2(N) \in \R^+.
$
Then, for all $x\in\D^n(R)$ with $0<R<S$,
\begin{equation}\label{eq:error_phi3}
\begin{split}
|&\phi_{\lambda}(x)-(P_N\phi_{\lambda})(x)|\\
& < M_2(N)\sqrt{\left(\dfrac{1}{1-(R/S)^2}\right)^n - \displaystyle\sum_{k=0}^{N}\left(\!\!{n\choose k}\!\!\right)\left(\dfrac{R}{S}\right)^{2k}}.
\end{split}
\end{equation}
\end{prop} 
\begin{proof}
We have
\begin{align*}
\left|\phi_{\lambda}- P_N\phi_{\lambda}\right|^2 &\leq \left(\displaystyle\sum_{|k|>N}|c_{k}||x|^k\right)^2,\\
 &\leq  \left(\displaystyle\sum_{|k|>N}|c_{k}|S^{k}\left(\dfrac{R}{S}\right)^{k}\right)^2,\\
 &\leq  
\left(\displaystyle\sum_{|k|>N}|c_k|^2S^{2k}\right)\left(\displaystyle\sum_{|k|>N}\left(\dfrac{R}{S}\right)^{2k}\right),\\
 &\leq  M_2(N)\left(\displaystyle\sum_{|k|>N}\left(\dfrac{R}{S}\right)^{2k}\right),
\end{align*}
where we used the Cauchy-Schwarz inequality. The rest of proof follows similarly as in the proof of Proposition \ref{thm:error_phi2}.
\end{proof}

\begin{remark}[Estimation of the domain of analyticity] An eigenfunction $\phi_\lambda$ is analytic over $\D^n(S)$ for some value $S\in\R^+_0$ which is unknown in practice. We can overcome this issue by relying on the Cauchy–Hadamard theorem, which ensures that $\phi_\lambda$ converges with radius $\rho:= (\rho_1,\ldots,\rho_n)$ if and only if $\limsup_{|k|\rightarrow\infty}\sqrt[|k|]{|c_k|\rho^{k}} = 1$. Note that, for one dimensional power series, the radius of convergence $\rho$ satisfies 
    \begin{equation}\label{radius_1D}
    \dfrac{1}{\rho} = \limsup_{k\rightarrow\infty}\left(|c_k|^{1/k}\right).
    \end{equation}
In order to estimate the radius of analyticity, we can compute the Taylor coefficients for $|k|< N$ with $N \gg 1$ and evaluate the quantity 
\begin{equation}\label{coeff_pratique}
\left|1-\max_{|k|\leq N}\sqrt[|k|]{|c_k|\rho^k} \right|
\end{equation} 
for different values of $\rho \in \R_+^n$. The radius of convergence is given by the value $\rho$ which minimizes \eqref{coeff_pratique} and, in particular, the radius $S$ of $\D^n(S)$ can be chosen as $S=\min_i\rho_i$.
\end{remark}

\section{Error bounds on the Lyapunov function and application to stability}\label{stab_sec} 

Propositions \ref{thm:error_phi1} - \ref{thm:error_phi3} characterize the approximation error on $\phi_{\lambda_i}$ for any $i=1,\ldots,n$. We now investigate how this error propagates to the approximation of the Lyapunov function $V$ and its time derivative $\dot{V}$. This will be useful to obtain an inner estimation of the basin of attraction of the equilibrium point in Section \ref{sec:application_stab}. 
\subsection{Error bound on $|V-\widetilde{V}|$}
We consider the general case where the eigenvalues $\lambda_i $ and eigenfunctions $\phi_{\lambda_i}$ are complex-valued for any $i\in\{1,\ldots,n\}$. Let $\Re(v)$ and $\Im(v)$ be respectively the real and imaginary part of a complex number $v\in\mathbb{C}$. For the sake of readability, we will denote $\Re(\phi_{\lambda_i}) = \phi_{\lambda_i}^R$, $\Im(\phi_{\lambda_i}) = \phi_{\lambda_i}^I$,  $\Re(\lambda_i) = \lambda_i^R$, and $\Im(\lambda_i) = \lambda_i^I$.  

The candidate Lyapunov function obtained with the approximate eigenfunctions is given by
\begin{equation}\label{V_approx}
\widetilde{V} = \sum_{i=1}^n|\widetilde{\phi}_{\,\widetilde{\lambda}_i}|^2 = \sum_{i=1}^n|P_N\phi_{\lambda_i}|^2.
\end{equation}
We have 
\setlength\arraycolsep{1pt}
\begin{align}
|V&-\widetilde{V}| = \Big|\sum_{i=1}^n\left(|\phi_{\lambda_i}|-|P_N\phi_{\lambda_i}|\right)\left(|\phi_{\lambda_i}|+|P_N\phi_{\lambda_i}|\right)\Big|,\nonumber\\
& \leq \sum_{i=1}^n\Big||\phi_{\lambda_i}|-|P_N\phi_{\lambda_i}|\Big|\Big||\phi_{\lambda_i}|-|P_N\phi_{\lambda_i}|+2|P_N\phi_{\lambda_i}|\Big|,\nonumber\\
&\leq 
\sum_{i=1}^n 2|P_N\phi_{\lambda_i}||\phi_{\lambda_i}-P_N\phi_{\lambda_i}|+|\phi_{\lambda_i}-P_N\phi_{\lambda_i}|^2\nonumber\\
&\triangleq \sum_{i=1}^n B_i(N).\label{error_bound_V} 
\end{align}

\noindent 
\begin{remark}\label{rem}
Assume that the eigenfunctions $\phi_{\lambda_i}$ are analytic over $\mathbb{D}^n(S_i)$ and satisfy the assumptions of Propositions 1, 2, or 3, with $R_i < S_i$. Then, the  we have that $|V-\widetilde{V}|$ converges to zero over $\mathbb{D}^n(R)$ as $N \rightarrow \infty$ where $R = \min_i R_i$ and $\mathbb{D}^n(S)$ with $S = \min_i S_i$ is the joint domain of analyticity.\end{remark}

\subsection{Error bound on $|\dot{V}-\dot{\widetilde{V}}|$} 

We first observe that  
\begin{equation}\label{dot_V_expansion}
\dot{V} = \mathcal{L}\left(\sum_{i=1}^n |\phi_{\lambda_i}|^2\right) = \sum_{i=1}^n \mathcal{L}|\phi_{\lambda_i}|^2 = 2\sum_{i=1}^n \lambda_i^R|\phi_{\lambda_i}|^2,
\end{equation}
where the last inequality follows from the fact that $\phi_{\lambda}$ is an eigenfunction of $\mathcal{L}$ \cite{c6}. Similarly, we have
\setlength\arraycolsep{1pt}
\begin{equation}
\label{dot_V_tilde_expansion}
\begin{split}
\dot{\widetilde{V}} &= \mathcal{L}\left(\displaystyle\sum_{i=1}^n |P_N\phi_{\lambda_i}|^2\right),\\
&= \displaystyle\sum_{i=1}^n \mathcal{L} \left(\Re(P_N\phi_{\lambda_i})^2+\Im(P_N\phi_{\lambda_i})^2\right),\\
& =  2\displaystyle\sum_{i=1}^n \left(P_N\phi^R_{\lambda_i}\mathcal{L}P_N\phi^R_{\lambda_i}+ P_N\phi^I_{\lambda_i}\mathcal{L}P_N\phi^I_{\lambda_i} \right).
\end{split}
\end{equation}
We are now in position to characterize the approximation error bound on the time derivative $\dot{V}$.
\begin{prop} Let $V$ and $\widetilde{V}$ be defined by \eqref{TrueLyap} and \eqref{V_approx} and assume that $\{\phi_{\lambda_i}\}_{i=1}^n$ are analytic over $\D^n(S)$. If $|P_N\phi^R_{\lambda_i}| < M_i$ and $|P_N\phi^I_{\lambda_i}| < K_i$ with $K_i, M_i > 0$ for all $i = 1,\ldots,n$, then 
\begin{align}
|\dot{V}-\dot{\widetilde{V}}| ~<~ &2\sum_{i=1}^n|\lambda_i^R|B_i(N) + M_i|(P_N-I)\mathcal{L}P_N\phi_{\lambda_i}^R|\nonumber\\ 
& \quad \quad + K_i|(P_N-I)\mathcal{L}P_N\phi_{\lambda_i}^I|.\label{error_dotV}  
\end{align}
\end{prop}
\begin{proof}
From \eqref{dot_V_expansion} and \eqref{dot_V_tilde_expansion}, we can write
\begin{align}
\addtolength{\jot}{2em}
\dot{V} - \dot{\widetilde{V}} &= 
2\sum_{i=1}^n \lambda_i^R|\phi_{\lambda_i}|^2-P_N\phi_{\lambda_i}^R\mathcal{L}P_N\phi_{\lambda_i}^R\nonumber\\
&\qquad -P_N\phi_{\lambda_i}^I\mathcal{L}P_N\phi_{\lambda_i}^I.\nonumber
\end{align}
By adding and subtracting the two terms $P_N\phi_{\lambda_i}^R(P_N\mathcal{L}P_N\phi_{\lambda_i}^R)$ and $P_N\phi_{\lambda_i}^I(P_N\mathcal{L}P_N\phi_{\lambda_i}^I)$, we obtain 
\begin{align}
\dot{V} - \dot{\widetilde{V}}& = 2\sum_{i=1}^n \lambda_i^R|\phi_{\lambda_i}|^2 + P_N\phi_{\lambda_i}^R(P_N\mathcal{L}P_N\phi_{\lambda_i}^R-\mathcal{L}P_N\phi_{\lambda_i}^R)\nonumber\\
&\qquad+P_N\phi_{\lambda_i}^I(P_N\mathcal{L}P_N\phi_{\lambda_i}^I-\mathcal{L}P_N\phi_{\lambda_i}^I)\nonumber\\[0.2cm]
& \qquad - P_N\phi_{\lambda_i}^R(P_N\mathcal{L}P_N\phi_{\lambda_i}^R) - P_N\phi_{\lambda_i}^I(P_N\mathcal{L}P_N\phi_{\lambda_i}^I).\nonumber
\end{align}
It follows from \eqref{eq:PLP} that $P_N\mathcal{L}P_N\phi_{\lambda_i} = P_N\lambda_i \phi_{\lambda_i}$ and taking the real and complex parts of this equality, we obtain 
\begin{eqnarray*}
    P_N\mathcal{L}P_N\phi_{\lambda_i}^R & = & \lambda_i^RP_N\phi_{\lambda_i}^R-\lambda_i^IP_N\phi_{\lambda_i}^I\\
    P_N\mathcal{L}P_N\phi_{\lambda_i}^I & = & \lambda_i^RP_N\phi_{\lambda_i}^I+\lambda_i^IP_N\phi_{\lambda_i}^R.
\end{eqnarray*}
We thus have 
\begin{align}
\dot{V} - \dot{\widetilde{V}} & = 2\sum_{i=1}^n \lambda_i^R|\phi_{\lambda_i}|^2 + P_N\phi_{\lambda_i}^R(P_N\mathcal{L}P_N\phi_{\lambda_i}^R-\mathcal{L}P_N\phi_{\lambda_i}^R)\nonumber\\
&\qquad +P_N\phi_{\lambda_i}^I(P_N\mathcal{L}P_N\phi_{\lambda_i}^I-\mathcal{L}P_N\phi_{\lambda_i}^I)\nonumber\\[0.2cm]
&\qquad-P_N\phi_{\lambda_i}^R(\lambda_i^RP_N\phi_{\lambda_i}^R-\lambda_i^IP_N\phi_{\lambda_i}^I)\nonumber\\[0.2cm]
&\qquad-P_N\phi_{\lambda_i}^I(\lambda_i^RP_N\phi_{\lambda_i}^I+\lambda_i^IP_N\phi_{\lambda_i}^R),\nonumber
\end{align}

\begin{align}
& = 2\sum_{i=1}^n \lambda_i^R(|\phi_{\lambda_i}|^2-|P_N\phi_{\lambda_i}|^2)\nonumber\\[0.1cm]
& \qquad + P_N\phi_{\lambda_i}^R(P_N\mathcal{L}P_N\phi_{\lambda_i}^R-\mathcal{L}P_N\phi_{\lambda_i}^R)\nonumber\\[0.2cm]
&\qquad+P_N\phi_{\lambda_i}^I(P_N\mathcal{L}P_N\phi_{\lambda_i}^I-\mathcal{L}P_N\phi_{\lambda_i}^I)\nonumber
\end{align}
Finally, the result follows from the triangle inequality and the definition of $B_i(N)$ in \eqref{error_bound_V}.
\end{proof}
\noindent We note that the two terms 
$$
|(P_N-I)\mathcal{L}P_N\phi_{\lambda_i}^R|
~\text{ and }~
|(P_N-I)\mathcal{L}P_N\phi_{\lambda_i}^I|
$$
appearing in \eqref{error_dotV} are known functions that can be computed for a fixed value $N$. Moreover, they correspond to the error between the analytic functions $\mathcal{L}P_N\phi_{\lambda_i}^R$ and $\mathcal{L}P_N\phi_{\lambda_i}^I$, respectively, and their truncation. Since this error converges to zero as $N \rightarrow \infty$, we have that $|\dot{V}-\dot{\widetilde{V}}|$ also converges to zero as $N \rightarrow \infty$ following similar lines as in Remark \ref{rem}.

\subsection{Estimation of the region of attraction}\label{sec:application_stab}

An inner approximation of the region of attraction of an equilibrium point can be computed as the largest level set of the candidate Lyapunov function $\dot{\widetilde{V}}$ that lie in the \textit{validity region} $\{x\in\X ~|~ \dot{\widetilde{V}} < 0\}$. Using the error bounds \eqref{error_bound_V} and \eqref{error_dotV} on $\widetilde{V}$ and $\dot{\widetilde{V}}$, we can derive sufficient conditions on the level set of $\widetilde{V}$ to guarantee a valid inner approximation of the region of attraction.
\begin{prop}\label{th_ROA} Let $V$ and $\widetilde{V}$ be defined by \eqref{TrueLyap} and \eqref{V_approx} over a region $\D^n(S)$ where $\{\phi_{\lambda_i}\}_{i=1}^n$  are analytic. If for all \mbox{$x\in D \subset \X$}, \mbox{$|V(x)-\widetilde{V}(x)|<\varepsilon_1$}, $|\dot{V}(x)-\dot{\widetilde{V}}(x)|<\varepsilon_2$ and
\begin{equation*}\label{ROA}
\Omega = \left\{x\in D \mid \gamma_1 < \widetilde{V}(x) < \gamma_2\right\}
\end{equation*} 
is not empty, then $\Omega$ is an inner approximation of the ROA of \eqref{eq} with $\gamma_1 = \left(-\varepsilon_2-2|\lambda_m|\varepsilon_1\right)/2\lambda_m$, $\lambda_m \triangleq \max_i \lambda_i^R$, and where $\gamma_2$ is the largest level set of $\widetilde{V}$ lying in $D$.
\end{prop}

\begin{proof}
It follows from \eqref{dot_V_expansion} that
$$
\dot{V} = 2\sum_{i=1}^n \lambda_i^R|\phi_{\lambda_i}|^2 < 2\lambda_m V
$$
and we have
\begin{equation*}
\begin{split}
\dot{\widetilde{V}} \leq  \dot{\widetilde{V}} - \dot{V} + 2\lambda_m V & = \dot{\widetilde{V}} - \dot{V} + 2\lambda_m(V-\widetilde{V}) + 2\lambda_m \widetilde{V} \\
& \leq \varepsilon_2 + 2|\lambda_m| \varepsilon_1 + 2\lambda_m \widetilde{V} .
\end{split}
\end{equation*}
\noindent Hence, $\dot{\widetilde{V}} < 0$ if $\widetilde{V} > \gamma_1 \triangleq (-\varepsilon_2-2|\lambda_m|\varepsilon_1)/2\lambda_m$.
Since $0\in E_{\gamma_1} :=\{x\in D~\mid~\widetilde{V}(x)<\gamma_1\}$, any trajectory in $\Omega$ is mapped to a neighborhood $E_{\gamma_1}$ of the origin, which concludes the proof.
\end{proof}
As $N\rightarrow \infty$, we have $\gamma_1 \rightarrow 0$ so that the convergence to the equilibrium is captured accurately. However, the value $\gamma_2$ converges to a constant value, and so does size of $\Omega$. Instead, a larger size of $\Omega$ can be obtained by increasing $R$, which will also increase $\gamma_1$ and therefore require an increase of $N$ for better accuracy.

\subsection{Alternative result with a surrogate system}

We can also obtain an alternative result to Proposition \ref{th_ROA} by considering a surrogate system. Assume now that there exists a vector field $\widetilde{F}$ such that  $P_N\phi_{\lambda_i}$ is an eigenfunction of $\widetilde{\mathcal{L}} = \widetilde{F} \cdot \nabla$ associated with the dynamics $\dot{x}=\widetilde{F}(x)$. In this case, the candidate $\widetilde{V}$ is a \textit{true} Lyapunov function for that surrogate system. We can then envision that the knowledge of an error bound on $|F-\widetilde{F}|$ is sufficient to provide stability guarantees for the original system $\dot{x}=F(x)$. This is summarized in the following proposition. 
\begin{prop}\label{vector_field}
Let $[\mathbf{J}_\phi(x)]_{ij} := \frac{\partial P_N\phi_{\lambda_i}}{\partial x_j}(x)$ and assume that $\mathbf{J}_\phi(x)$ is invertible for any $x\in D \subset \X$. Then, $P_N\phi_{\lambda_i}$ is an eigenfunction of $\widetilde{\mathcal{L}} = \widetilde{F} \cdot \nabla$ with 
$$\widetilde{F}(x) = \mathbf{J}_\phi(x)^{-1}[\lambda_1P_N\phi_{\lambda_1}(x),\ldots,\lambda_nP_N\phi_{\lambda_n}(x)]^{\top}.$$ Moreover, if $|F_i-\widetilde{F}_i|(x) < \varepsilon_i$, $\max_{x\in D}\left|\frac{\partial \widetilde{V}}{\partial x_i}\right|<\delta_i$ with $\varepsilon_i,\delta_i > 0$ for any $i = 1,\ldots,n$, and
$$
\Omega = \left\{x\in D \mid \sum_i \dfrac{\varepsilon_i\delta_i}{2|\lambda_m|} < \widetilde{V}(x) < \gamma_2\right\}
$$
is not empty, then $\Omega$ is an inner approximation of the ROA of \eqref{eq} with $\lambda_m = \max_i \lambda_i^R$ and $\gamma_2$ the largest level set of $\widetilde{V}$ lying in $D$.

\end{prop}

\begin{proof} For any $i\in\{1,\ldots,n\}$, $P_N\phi_{\lambda_i}$ is an eigenfunction of $\widetilde{\mathcal{L}}$ if $\widetilde{\mathcal{L}}P_N\phi_{\lambda_i} = \lambda_iP_N\phi_{\lambda_i}$. Thus, $\widetilde{F}$ should satisfy 
\begin{equation}\label{F^*}
\mathbf{J}_\phi(x)\widetilde{F} = 
[\lambda_1P_N\phi_{\lambda_1}(x),\ldots,\lambda_nP_N\phi_{\lambda_n}(x)]^{\top}
\end{equation}
Since $\mathbf{J}_\phi(x)$ is invertible, this completes the first part of the proof. Next, we observe that 
\begin{equation*}
\begin{split}
\dot{\widetilde{V}} = \nabla \widetilde{V}^{\top} F & = \nabla \widetilde{V}^{\top} F + \nabla \widetilde{V}^{\top} \widetilde{F} - \nabla \widetilde{V}^{\top} \widetilde{F} \\
& \leq \nabla \widetilde{V}^{\top}\left(F-\widetilde{F}\right) + 2\lambda_m\widetilde{V},\\
 & \leq  \sum_{i=1}^n \left|\dfrac{\partial \widetilde{V}}{\partial x_i}\right||F_i-F_i^*| + 2\lambda_m\widetilde{V},\\
  & \leq   \sum_{i=1}^n \delta_i\varepsilon_i + 2
  \lambda_m\widetilde{V}.
\end{split}
\end{equation*}

\noindent Then, $\dot{\widetilde{V}} < 0$ if $\widetilde{V} > \gamma_1 = \sum_i (\varepsilon_i\delta_i)/2|\lambda_m|$. 
Since $0\in E_{\gamma_1} :=\{x\in D~\mid~\widetilde{V}(x)<\gamma_1\}$, any trajectory in $\Omega$ is mapped to a neighborhood $E_{\gamma_1}$ of the origin, which concludes the proof.
\end{proof}

We note that the origin is an equilibrium point of $\widetilde{F}$ since $P_N\phi_{\lambda_i}(0) = 0$ for any $N\in\N_0$ and $i\in\{1,\ldots,n\}$. Moreover, it is characterized by the same local stability property as the equilibrium of $\widetilde{F}$, as shown in the following result.
\begin{prop} Let $\mathbf{J}_{{F}}(x)$ and $\mathbf{J}_{\widetilde{F}}(x)$ be the Jacobian matrices associated with $F$ and $\widetilde{F}$, respectively. If $\lambda$ is an eigenvalue of $\mathbf{J}_{{F}}(0)$, then it is also an eigenvalue of $\mathbf{J}_{\widetilde{F}}(0)$. 
\end{prop}

\begin{proof}
For any eigenvalue $\lambda_i$ of $\mathbf{J}_{{F}}(0)$, it follows from \eqref{F^*} that 
$$
\sum_{k=1}^n\widetilde{F}_k\dfrac{\partial P_N\phi_{\lambda_i}}{\partial x_k} = \lambda_i P_N\phi_{\lambda_i}. 
$$

\noindent Differentiating this equality with respect to $x_j$ leads to
$$
\sum_{k=1}^n\dfrac{\partial \widetilde{F}_k}{\partial x_j}\dfrac{\partial P_N\phi_{\lambda_i}}{\partial x_k} + \widetilde{F}_k\dfrac{\partial^2 P_N\phi_{\lambda_i}}{\partial x_j\partial x_k} = \lambda_i \dfrac{\partial P_N\phi_{\lambda_i}}{\partial x_j}
$$
and evaluating at $0$ with $\widetilde{F}(0)=0$ yields \mbox{$\mathbf{J}_{\widetilde{F}}(0)^{\top} v_i = \lambda_i v_i$} where $v_i$ is the $i$th column of $\mathbf{J}_{\phi}(0)$.
This implies that $\lambda_i$ is an eigenvalue of $\mathbf{J}_{\widetilde{F}}(0)$. Note that the vectors $v_i$ are nonzero since $\mathbf{J}_{{\phi}}(0)$ is invertible.
\end{proof}
The above result is appealing since no information is required on the eigenfunction. However, the existence of the surrogate system might potentially suffer from invertibility issues related to the matrix $\mathbf{J}_{\phi}$. This requires further investigation that is left for future research.

\section{Numerical Application}\label{example}

In this section, we use the result of Proposition \ref{th_ROA} to estimate the region of attraction of two dynamical systems. For the sake of conciseness, we will only consider error bounds on the eigenfunctions from Propositions \ref{thm:error_phi1} and \ref{thm:error_phi2}. 

\paragraph{Example 1}

For illustrative purpose, we first consider the system $\dot{x} = -x+2x^2$ for $x\in[-1,1]$, which admits the equilibrium $x^*=0$ with the region of attraction $[-1,1/2[$. The only eigenvalue of $\mathbf{J}_F(0)$ is $\lambda = -1$ so there exists only one principal eigenfunction $\phi_{\lambda}$ which is real. The Taylor coefficients $c_k$ of $\phi_{\lambda}$ have been computed by exploiting the triangular structure of $\mathbf{L}_\ell$.  

\begin{figure}[b!]
\centering
\vspace{-0.4cm}
\subfigure[]{
\hspace{-0.4cm}
    \label{Fig1a}   \includegraphics[width=0.2\textwidth]{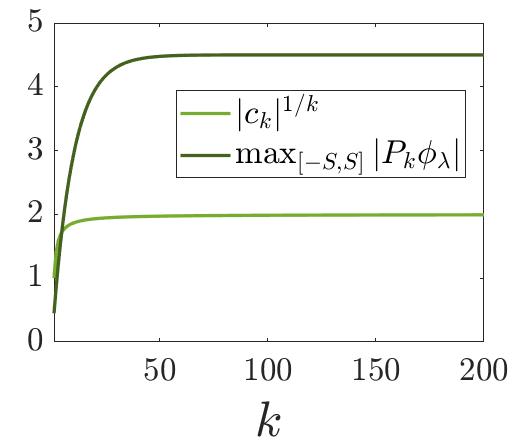}
}~\subfigure[]{
    \label{Fig1b}
\includegraphics[width=0.2\textwidth]{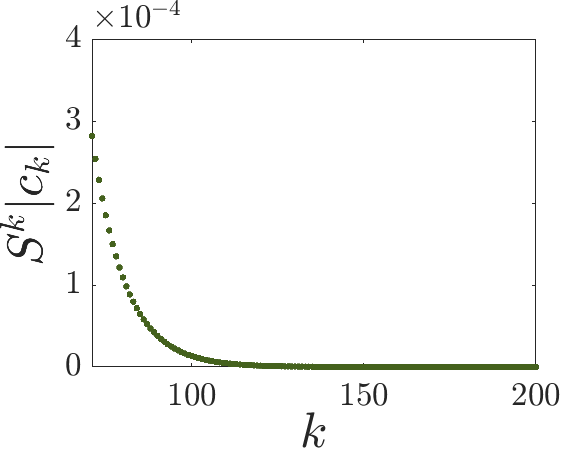}
}\vspace{-0.2cm}
\subfigure[]{\label{Fig1c}
\includegraphics[width=0.4\textwidth]{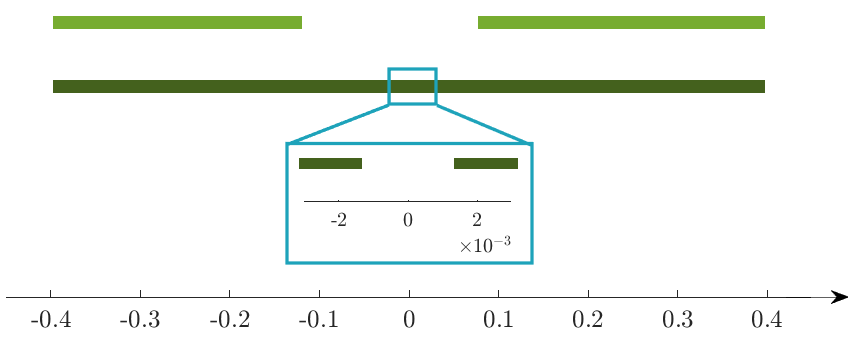}}
\caption{\subref{Fig1a} \textbf{Light green}: The convergence of the sequence $|c_k|^{1/k}$ to the value $2$ indicates that $\phi_{\lambda}$ is analytic over $\D^1(1/2)$. \textbf{Dark green}: The sequence $(|P_k\phi_{\lambda}|)_k$ is upper bounded by the value $5$. \subref{Fig1b} The decreasing sequence $|c_k|S^k$ allows to set $\max_{k>70}|c_{k}|S^{k} < 4\times10^{-4}$. \subref{Fig1c} An inner approximation of the ROA is computed with $R = 0.39$ and with the bounds obtained from Proposition \eqref{thm:error_phi1} (light green) and Proposition \eqref{thm:error_phi2} (dark green).}
\label{Figure1}
\end{figure}

\noindent As shown in Figure \ref{Figure1} \subref{Fig1a} (light green), the sequence $|c_k|^{1/k}$ converges to $2$, which indicates a radius of convergence $\rho=1/2$, according to \eqref{radius_1D}. Here, we set $S = 0.45$ and $R = 0.39$. Moreover, an error bound on $|\phi_{\lambda}|$ over $\D^1(S)$ is computed from its Taylor series for $N \gg 1$. In particular, an upper bound of the sequence $(|P_k\phi_{\lambda}|)_k$ indicates that $\phi_{\lambda} < 5$ (Figure \ref{Figure1} \subref{Fig1a} (dark green)). For $N = 70$, we obtain the error bound $|\phi_{\lambda}-P_N\phi_{\lambda}| < 0.0015 $ from Proposition \ref{thm:error_phi1}. Alternatively, a bound on $\max_{k > N}|c_k|S^{k}$ is estimated to $\max_{k > N}|c_k|S^{k} < 4 \times 10^{-4}$ (see Figure \ref{Figure1} \subref{Fig1b}) and we obtain the error bound $|\phi_{\lambda}-P_N\phi_{\lambda}| < 1.16\times 10^{-7} $ from Proposition \ref{thm:error_phi2}. Note that, in this one-dimensional case, the eigenfunction can be computed exactly and is given by 
$$
\phi_{\lambda}(x) = \dfrac{x}{1-2x} =\sum_{k\geq 1}2^{k-1}x^{k}.
$$
We can easily verify that the series converges as \mbox{$|x|<1/2$} and the Taylor coefficients satisfy the bounds obtained above from Figure \ref{Figure1} \subref{Fig1a} and \subref{Fig1b}.
Figure \ref{Figure1} \subref{Fig1c} shows inner approximations of the ROA obtained with Proposition \ref{th_ROA} for $D = \D^n(R)$. The approximations correspond to the set $\{x\in R ~|~ \widetilde{V}(x) > \gamma_1\}$ where $\gamma_1$ has been computed by using the error bound on $|\phi_{\lambda_1}-P_N\phi_{\lambda_1}|$ from Propositions \ref{thm:error_phi1} (light green) and \ref{thm:error_phi2} (dark green). Note that Proposition \ref{thm:error_phi2} provides a better approximation since the bound is less conservative. In both cases, the right boundary of the ROA is properly captured, while the approximation is conservative at the left boundary. This is due to the radius of convergence of $\phi_{\lambda}$, which restricts the approximation to $\D^n(R)$.

\begin{figure}[b!]
\centering
\vspace{-0.4cm}
\subfigure[]{
    \label{Fig2a}   \hspace{-0.2cm}\includegraphics[width=0.2\textwidth]{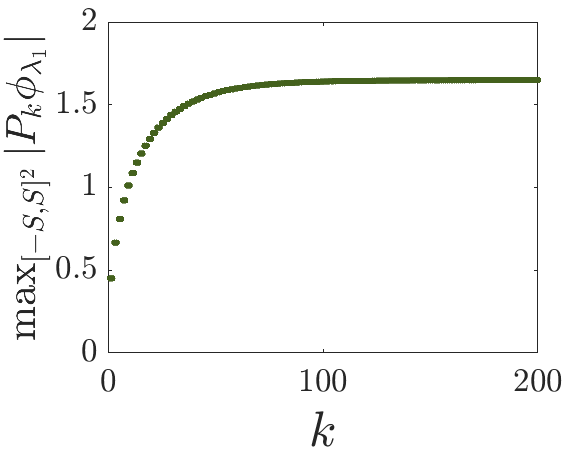}
}~\subfigure[]{
    \label{Fig2b}
\includegraphics[width=0.2\textwidth]{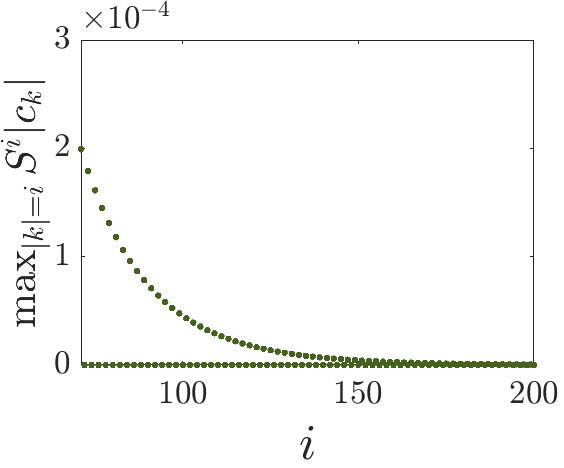}
}\vspace{-0.3cm}
\subfigure[]{\hspace{-0.6cm}\label{Fig2c}
\includegraphics[width=0.2\textwidth]{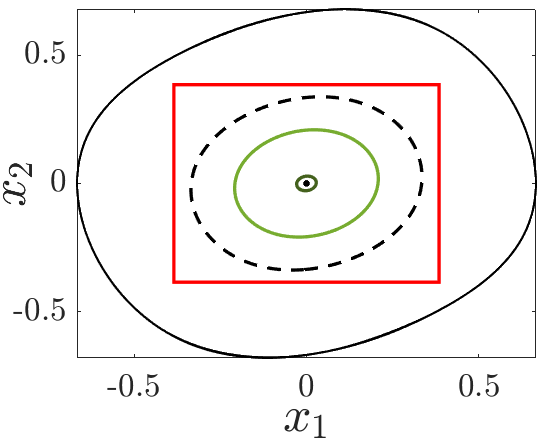}}
\caption{\subref{Fig2a} The sequence $(|P_k\phi_{\lambda}|)_k$ is upper bounded by the value $2$. \subref{Fig2b} The decreasing sequence of the nonzero coefficients $\max_{|k|=i}S^i|c_k|$ allows to set $\max_{|k|>70}|c_{k}|S^{k} < 3\times10^{-4}$. \subref{Fig2c} An inner approximation of the ROA is computed with $R = 0.385$ and with the bounds obtained from Proposition \eqref{thm:error_phi1} (light green) and Proposition \eqref{thm:error_phi2} (dark green). The dotted black curve is the largest level set $\widetilde{V}=\gamma_2$ inside $\D^n(R)$ (red box) while the full black curve is the true ROA of the system.} 
\label{Figure2}
\end{figure}

\paragraph{Example 2} We consider the Van der Pol dynamics 
$$
\left\{
\begin{array}{rcl}
\dot{x}_1 & = & -x_2,\\[0.2cm]
\dot{x}_2 & = & -\mu(1-9x_1^2)x_2 + x_1,
\end{array}
\right.
$$
where $\mu = 1/5$ and $[x_1,x_2]\in[-1,1]^2$. The system admits a stable equilibrium at the origin, whose region of attraction is bounded by an unstable limit cycle. Principal Koopman eigenfunctions are associated with two complex conjugate eigenvalues of $\mathbf{J}_F(0)$ and we have $V = 2|\phi_{\lambda_1}|^2$. The quantity \eqref{coeff_pratique} related to $\phi_{\lambda_1}$ is computed for different values $\rho \in[0,1]^2$ and the minimizer is $(0.4550,0.4690)$. We thus set $S = 0.43$ and $R = 0.385$. A bound on $|\phi_{\lambda_1}|$ and $\max_{|k| > N}|c_k|S^{k}$ are computed along similar lines as for Example 1. According to Figure \ref{Figure2} \subref{Fig2a} and \subref{Fig2b}, we set \mbox{$|\phi_{\lambda_1}| < 2$} and $\max_{|k| > N}|c_k|S^{k} < 3 \times 10^{-4}$ for $N = 70$ where the latter bound was computed by evaluating the maximum value of the sequence $(\max_{|k| = i}|c_k|S^i)_{70<i<200}$. Next, we obtain the error bounds $|\phi_{\lambda}-P_N\phi_{\lambda}| < 0.0075 $ and $|\phi_{\lambda}-P_N\phi_{\lambda}| < 9.01\times 10^{-5}$ from Proposition \ref{thm:error_phi1} and \ref{thm:error_phi2}, respectively.

Figure \ref{Figure2} \subref{Fig2c} shows approximations of the ROA obtained with Proposition \ref{th_ROA} for $D = \D^n(R)$. The approximations correspond to the set $\{x\in R ~|~ \widetilde{V}(x) > \gamma_1\}$ where $\gamma_1$ has been computed by using the error bound on $|\phi_{\lambda_1}-P_N\phi_{\lambda_1}|$ from Propositions \ref{thm:error_phi1} (light green) and \ref{thm:error_phi2} (dark green). The dotted black line is the largest level set $\widetilde{V}=\gamma_2$ inside $\D^n(R)$. As in the first example, we observe a better approximation with \eqref{eq:error_phi2}. Again, the approximation is restricted to $D = \D^n(R)$ due to the radius of convergence of the Taylor expansion of the eigenfunction.

\section{Conclusions and Perspectives}\label{conclusion}

In this paper, we have developed error bounds on the approximation of a Lyapunov function and its time derivative. We assume that the Lyapunov function is obtained from analytic eigenfunctions that are approximated by truncated Taylor series. The error bounds were leveraged to obtain a rigorous inner approximation of the region of attraction, which is solely based on sufficient conditions on the level set of the Lyapunov function candidate. Finally, our results were illustrated with two examples.

The obtained results appear to be conservative, mostly because they are restricted by the radius of convergence of the eigenfunctions. This limitation could be overcome by considering Taylor expansions at other states than the equilibrium, or by considering other sets of basis functions that do not rely on the underlying analyticity assumption. Finally, the potential of the surrogate system (Proposition \ref{vector_field}) could be further investigated from a numerical point of view.   

\addtolength{\textheight}{-12cm}

\end{document}